\documentclass[12pt]{article}
\usepackage{amssymb}
\usepackage{amsthm}
\usepackage{amsmath}
\usepackage{graphicx}
\usepackage{enumerate}

\DeclareMathOperator{\Con}{Con}

\DeclareMathOperator{\DM}{DM}
\DeclareMathOperator{\BDM}{\mathbf{DM}}

\newtheorem{theorem}{Theorem}[section]
\newtheorem{definition}[theorem]{Definition}
\newtheorem{lemma}[theorem]{Lemma}

\newtheorem{corollary}[theorem]{Corollary}
\title{Logical and algebraic properties of generalized orthomodular posets}
\author{Ivan~Chajda and Helmut~L\"anger}
\date{}
\begin{document}

\footnotetext{Support of the research of the authors by the Austrian Science Fund (FWF), project I~4579-N, and the Czech Science Foundation (GA\v CR), project 20-09869L, entitled ``The many facets of orthomodularity'', as well as by \"OAD, project CZ~02/2019, entitled ``Function algebras and ordered structures related to logic and data fusion'', and, concerning the first author, by IGA, project P\v rF~2020~014, is gratefully acknowledged.}

\maketitle

\begin{abstract}
Generalized orthomodular posets were introduced recently by D.~Fazio, A.~Ledda and the first author of the present paper in order to establish a useful tool for studying the logic of quantum mechanics. They investigated structural properties of these posets. In the present paper we study logical and algebraic properties of these posets. In particular, we investigate conditions under which they can be converted into operator residuated structures. Further, we study their representation by means of algebras (directoids) with everywhere defined operations. We prove congruence properties for the class of algebras assigned to generalized orthomodular posets and, in particular, for a subvariety of this class determined by a simple identity. Finally, in contrast to the fact that the Dedekind-MacNeille completion of an orthomodular poset need not be an orthomodular lattice we show that the Dedekind-MacNeille completion of a stronger version of a generalized orthomodular poset is nearly an orthomodular lattice. 

\end{abstract}



\section{Introduction}

Although the logic of quantum mechanics was axiomatized by K.~Husimi (\cite H) and G.~Birkhoff and J.~von~Neumann (\cite{Bv}) by means of orthomodular lattices, it was early shown that this description need not be appropriate in all concerns. Orthomodular lattices characterize the lattice of projection operators on a Hilbert space. In 1963, many years after orthomodular lattices have been introduced, it was realized that a more appropriate formalization of the logic of quantum mechanics could be obtained by replacing the axioms of orthomodular lattices by the weaker axioms of orthomodular posets, see e.g.\ \cite M. The reason for this weakening was that in the logic of quantum mechanics the disjunction of two propositions may exist only in the case that they are comparable with each other or orthogonal to each other. On the other hand, also this assumption may be too restrictive. Moreover, orthomodular posets cannot be extended to orthomodular lattices by means of the Dedekind-MacNeille completion. Therefore the concept of {\em generalized orthomodular posets} was introduced in \cite{CFL}. In that paper the order-theoretical properties of generalized orthomodular posets were investigated.

Since all so-called quantum structures, such as orthomodular lattices, orthomodular posets, generalized orthomodular posets (see also \cite{CL18a} and \cite{CL18b}), are assumed to be an algebraic axiomatization of the semantics of the logic of quantum mechanics, the natural question arises how the logical connective of implication should be modeled within these logics. Implication turns out to be one of the most fundamental and most productive logical connectives which enables logical deduction and therefore it should be introduced in a way acceptable in logics. Usually, implication is considered to be sound if it is related with conjunction via the so-called {\em adjointness}, i.e.\ implication and conjunction should form a {\em residuated pair}. For orthomodular lattices this task was solved by the authors in \cite{CL17a} and \cite{CL17b}, for orthomodular posets in \cite{CL14} and for some more general posets in \cite{CL18b}.

The aim of the present paper is to describe some algebraic and logical properties of generalized orthomodular posets, to show how these posets can be represented by means of algebras with everywhere defined operations and to introduce certain modifications of the connective of implication related to certain types of conjunction via a generalized version of adjointness.

A certain generalization of orthomodular posets avoiding existential quantifiers was introduced in \cite{CFL} under the name {\em pseudo-orthomodular poset} where it was shown that though such a poset cannot be organized into a residuated structure, it is possible to define binary operators $M(x,y)$ and $R(x,y)$ on such a poset satisfying the so-called operator left adjointness
\begin{enumerate}
\item[(1)] $M(x,y)\subseteq L(z)$ if and only if $L(x)\subseteq R(y,z)$.
\end{enumerate}
This motivated us to find something analogous for generalized orthomodular posets. Since the definition of a generalized orthomodular poset is a bit more simple than that of a pseudo-orthomodular one, we will need only one operator, namely $R(x,y)$. On the other hand, analogously as in \cite{CL14}, we need an additional condition guaranteeing property (1). Since our operator $M(x,y)$ will be commutative, instead of operator left adjointness we will have conditional operator adjointness. Further, we will show that if a strong generalized orthomodular poset is considered then the conjunction can be defined in a slightly different way as well as the residuated operator $R(x,y)$ such that we really obtain a left residuated structure.

\section{Basic concepts}

We start by defining the aforementioned concepts.

Let $(P,\leq)$ be a poset, $a,b\in P$ and $A,B\subseteq P$. Then $A\leq B$ should mean $x\leq y$ for all $x\in A$ and $y\in B$. Instead of $\{a\}\leq B$, $A\leq\{b\}$ and $\{a\}\leq\{b\}$ we simply write $a\leq B$, $A\leq b$ and $a\leq b$, respectively. Moreover, we define
\begin{align*}
L(A) & :=\{x\in P\mid x\leq A\}, \\
U(A) & :=\{x\in P\mid A\leq x\}
\end{align*}
and call these sets the {\em lower} and {\em upper cone} of $A$, respectively. Instead of $L(\{a,b\})$, $L(\{a\}\cup B)$, $L(A\cup B)$ and $L(U(A))$ we simply write $L(a,b)$, $L(a,B)$, $L(A,B)$ and $LU(A)$, respectively. Analogously we proceed in similar cases. 

Let $(P,\leq,{}')$ be a poset with a unary operation $'$ and $A\subseteq P$. We define $A':=\{x'\mid x\in A\}$. We call $'$ an {\em antitone involution} of $(P,\leq)$ if both $x''\approx x$ and if for $x,y\in P$, $x\leq y$ implies $y'\leq x'$. Let $(P,\leq,{}',0,1)$ be a bounded poset with a unary operation $'$. We call $'$ a {\em complementation} on $(P,\leq,0,1)$ if $L(x,x')\approx\{0\}$ and $U(x,x')\approx\{1\}$.

An {\em orthoposet} is a bounded poset $(P,\leq,{}',0,1)$ with an antitone involution $'$ which is a complementation.

Any orthoposet $(P,\leq,{}',0,1)$ satisfies the {\em De Morgan's laws}
\begin{align*}
(L(x,y))' & \approx U(x',y'), \\
(U(x,y))' & \approx L(x',y').
\end{align*}
The following concept was introduced in \cite{CFL}.

\begin{definition}
A {\em generalized orthomodular poset} is an orthoposet $(P,\leq,{}',0,1)$ satisfying the condition
\begin{enumerate}
\item[{\rm(2)}] $x\leq y$ implies $U(y)=U(x,L(x',y))$.
\end{enumerate}
\end{definition}

Using De Morgan's laws it is elementary to prove that {\rm(2)} is equivalent to
\begin{enumerate}
\item[(3)] $x\leq y$ implies $L(x)=L(y,U(x,y'))$.
\end{enumerate}
Since $y$ is the smallest element of $U(y)$, $U(y)=U(x,L(x',y))$ means that $y$ is the smallest element of $U(x,L(x',y))$, i.e.\ $y$ is the smallest upper bound, i.e.\ the supremum, of $\{x\}\cup L(x',y)$ which means $y=x\vee L(x',y)$. Hence {\rm(2)} can be written in the form
\[
x\leq y\text{ implies }y=x\vee L(x',y)).
\]
Analogously, {\rm(3)} can be written in the form
\[
x\leq y\text{ implies }x=y\wedge U(x,y').
\]
Now we define also a stronger version of a generalized orthomodular poset as follows.

\begin{definition}
A {\em strong generalized orthomodular poset} is an orthoposet $(P,\leq,{}',0,1)$ satisfying the condition that for all $x\in P$ and for all subsets $B$ of $P$
\begin{enumerate}
\item[{\rm(4)}] $x\leq U(B)\text{ implies }U(B)=U(x,L(x',U(B)))$.
\end{enumerate}
\end{definition}

Using De Morgan's laws it is elementary to prove that this is equivalent to the condition that for all $y\in P$ and $A\subseteq P$
\begin{enumerate}
\item[(5)] $L(A)\leq y\text{ implies }L(A)=L(y,U(L(A),y'))$.
\end{enumerate}
Obviously, every strong generalized orthomodular poset is a generalized orthomodular poset.

Now we recall the useful concept of a directoid which serves as an algebraization of a given poset. In Section~4 we will show how a generalized orthomodular poset can be converted into an algebra with everywhere defined operations by using of an assigned directoid.

A {\em {\rm(}join-{\rm)}directoid} (see \cite{CL11} and \cite{JQ}) is a groupoid $(D,\sqcup)$ satisfying the following identities:
\begin{enumerate}[(i)]
\item $x\sqcup x\approx x$ (idempotency),
\item $x\sqcup y\approx y\sqcup x$ (commutativity),
\item $x\sqcup((x\sqcup y)\sqcup z)\approx(x\sqcup y)\sqcup z$ (weak associativity).
\end{enumerate}
Let $\mathbf P=(P,\leq)$ be a poset. A groupoid $(P,\sqcup)$ is called a {\em directoid assigned to $\mathbf P$} if it satisfies the following conditions for all $x,y\in P$:
\begin{itemize}
\item $x\sqcup y=y$ if $x\leq y$,
\item $x\sqcup y=y\sqcup x\in U(x,y)$.
\end{itemize}
Assume $\mathbb D(\mathbf P)$ to be a directoid assigned to $\mathbf P$. Then $\mathbb D(\mathbf P)$ is a directoid by the above definition. Conversely, let $\mathbf D=(D,\sqcup)$ be a directoid and let $\mathbb P(\mathbf D)$ the ordered pair $(D,\leq)$ where $\leq$ denotes the binary relation on $D$ defined by
\[
x\leq y\text{ if }x\sqcup y=y,
\]
the so-called {\em induced order}. Then $\mathbb P(\mathbf D)$ is a poset and $\mathbb P(\mathbb D(\mathbf P))=\mathbf P$. This shows that, though $\mathbb D(\mathbf P)$ is in general not uniquely determined by $\mathbf P$, it contains the whole information on $\mathbf P$. If $\mathbf P=(P,\leq,0,1)$ is a bounded poset then there exists an assigned directoid since $1\in U(x,y)$ and hence $U(x,y)\neq\emptyset$ for all $x,y\in P$. One can easily check that an assigned directoid $\mathbb D(\mathbf P)=(P,\sqcup,0,1)$ satisfies the identities
\[
x\sqcup0\approx x\text{ and }x\sqcup1\approx1.
\]
Such a directoid will be referred to as a {\em bounded directoid}.

\section{Residuation}

Now we introduce one of our main concepts.

\begin{definition}\label{def1}
A {\em conditionally operator residuated poset} is an ordered six-tuple $\mathbf R=(P,\leq,{}',R,0,1)$ such that $(P,\leq,{}',0,1)$ is a bounded poset with a unary antitone operation $'$ and $R$ is a mapping from $P^2$ to $2^P$ satisfying the following conditions for all $x,y,z\in P$:
\begin{enumerate}[{\rm(i)}]
\item If $x'\leq y$ then $L(x,y)\subseteq L(z)$ implies $L(x)\subseteq R(y,z)$,
\item if $z\leq y$ then $L(x)\subseteq R(y,z)$ implies $L(x,y)\subseteq L(z)$,
\item $R(x,0)\approx L(x')$,
\item $R(x'',x)\approx P$.
\end{enumerate}
$\mathbf R$ is said to satisfy {\em operator divisibility} if
\[
x\leq y\text{ implies }L(y,U(R(y,x)))=L(x).
\]
\end{definition}

In the sequel we will show that there are close connections between generalized orthomodular posets and conditionally operator residuated posets.

\begin{theorem}
Let $(P,\leq,{}',0,1)$ be a generalized orthomodular poset and put
\[
R(x,y):=LU(x',y)\text{ for all }x,y\in P.
\]
Then $\mathbf R:=(P,\leq,{}',R,0,1)$ is a conditionally operator residuated poset satisfying operator divisibility.
\end{theorem}

\begin{proof}
Let $a,b,c\in P$.
\begin{enumerate}[(i)]
\item If $a'\leq b$ and $L(a,b)\subseteq L(c)$ then $b'\leq a$ and using (2), we compute
\[
L(a)=LU(a)=LU(b',L(a,b))\subseteq LU(b',L(c))=LU(b',c)=R(b,c).
\]
\item If $c\leq b$ and $L(a)\subseteq R(b,c)$ then, using (3), we derive
\[
L(a,b)=L(a)\cap L(b)\subseteq LU(b',c)\cap L(b)=L(U(b',c),b)=L(c).
\]
\item Further, we have $R(x,0)\approx LU(x',0)\approx LU(x')\approx L(x')$.
\item We have $R(x'',x)\approx LU(x''',x)\approx LU(x',x)\approx L(1)=P$.
\end{enumerate}
In case $a\leq b$ we finally have
\[
L(b,U(R(b,a)))=L(b,ULU(b',a))=L(b,U(b',a))=L(a).
\]
Thus $\mathbf R$ satisfies operator divisibility.
\end{proof}

Hence we have shown that every generalized orthomodular poset can be organized into a conditionally operator residuated poset in analogy to the fact that every orthomodular poset can be converted into a conditionally residuated poset (cf.\ \cite{CL14}). Let us note that a generalized orthomodular poset can be reduced to an orthomodular poset if it is orthogonal (see \cite C for this concept and its properties).

Now we are interested in the converse question, i.e.\ whether a conditionally operator residuated poset is in fact a generalized orthomodular poset. In the next theorem we show that this is the case if the unary operation is antitone and satisfies operator divisibility.

\begin{theorem}
Let $(P,\leq,{}',R,0,1)$ be a conditionally operator residuated poset satisfying operator divisibility and assume
\[
R(x,y)=LU(x',y)\text{ for all }x,y\in P.
\] 
Then $\mathbf P:=(P,\leq,{}',0,1)$ is a generalized orthomodular poset.
\end{theorem}

\begin{proof}
Let $a,b\in P$. From $L(a')\subseteq R(a,0)$ we obtain $L(a',a)\subseteq L(0)$, i.e.\ $L(a,a')=\{0\}$. From $L(a,a')\subseteq L(0)$ we obtain $a\in L(a)\subseteq R(a',0)=L(a'')$ by (iii) which yields $a\leq a''$. Conversely, we apply (iv) and from $L(a'')\subseteq R(a'',a)$ we obtain $a''\in L(a'',a'')\subseteq L(a)$ which yields $a''\leq a$. Together we have $a''=a$. Thus $'$ is an involution. Hence we can apply De Morgan's laws to $L(a,a')=\{0\}$ in order to obtain $U(a,a')=\{1\}$, i.e.\ $'$ is also a complementation on $(P,\leq)$ and hence $\mathbf P$ an orthoposet. If $a\leq b$ then using De Morgan's laws and operator divisibility we finally obtain
\begin{align*}
U(a,L(a',b)) & =(L(a',U(a,b')))'=(L(a',ULU(a,b')))'=(L(a',U(R(a',b'))))'= \\
             & =(L(b'))'=U(b).
\end{align*}
Altogether, $\mathbf P$ is a generalized orthomodular poset.
\end{proof}

It is a natural question whether we can obtain also a structure which is residuated in a broader sense but no additional conditions must be supposed.

In what follows, we show that strong generalized orthomodular posets can be converted into left residuated structures where both the operators $M(x,y)$ and $R(x,y)$ (i.e.\ conjunction and implication in a broad sense) are everywhere defined. For this, we modify Definition~\ref{def1} as follows.

\begin{definition}\label{def2}
An {\em operator residuated poset} is an ordered seven-tuple $\mathbf R=(P,\leq,{}',M,R,0,1)$ such that $(P,\leq,{}',0,1)$ is a bounded poset with a unary antitone operation $'$ and $M$ and $R$ are mappings from $P^2$ to $2^P$ satisfying the following conditions for all $x,y,z\in P$:
\begin{enumerate}[{\rm(i)}]
\item $M(x,y)\subseteq L(z)$ if and only if $L(x)\subseteq R(y,z)$ {\rm(}{\em operator adjointness}{\rm)},
\item $R(x,0)\approx L(x')$,
\item $R(x,x'')\approx R(x'',x)\approx P$.
\end{enumerate}
$\mathbf R$ is said to satisfy {\em operator divisibility} if
\[
x\leq y\text{ implies }L(y,U(R(y,x)))=L(x).
\]
\end{definition}

Now we can show that strong generalized orthomodular posets can be organized into left residuated structures analogously as it was done for modular lattices and strongly modular posets in \cite{CL19a} and \cite{CL19b}.

\begin{theorem}
Let $(P,\leq,{}',0,1)$ be a strong generalized orthomodular poset and put
\begin{align*}
M(x,y) & :=L(U(x,y'),y), \\
R(x,y) & :=LU(x',L(x,y))
\end{align*}
for all $x,y\in P$. Then $\mathbf R:=(P,\leq,{}',M,R,0,1)$ is an operator residuated poset satisfying operator divisibility.
\end{theorem}

\begin{proof}
Let $a,b,c\in P$.
\begin{enumerate}[(i)]
\item If $M(a,b)\subseteq L(c)$ then, using (4), we compute
\begin{align*}
L(a) & =LU(a)\subseteq LU(a,b')=LU(b',L(U(a,b'),b))=LU(b',L(b)\cap L(U(a,b'),b))\subseteq \\
     & \subseteq LU(b',L(b)\cap L(c))=LU(b',L(b,c))=R(b,c).
\end{align*}
If, conversely, $L(a)\subseteq R(b,c)$ then
\[
U(b',L(b,c))=ULU(b',L(b,c))=U(R(b,c))\subseteq UL(a)=U(a)
\]
and hence, using (5), we obtain
\begin{align*}
M(a,b) & =L(U(a,b'),b)=L(U(a)\cap U(b'),b)\subseteq L(U(b',L(b,c))\cap U(b'),b)= \\
       & =L(U(b',L(b,c)),b)=L(b,c)\subseteq L(c).
\end{align*}
It is easy to verify the remaining conditions from Definition~\ref{def2}.
\item We have $R(x,0)\approx LU(x',L(x,0))\approx LU(x')\approx L(x')$.
\item We have $R(x,x'')\approx R(x'',x)\approx R(x,x)\approx LU(x',L(x,x))\approx LU(x',x)\approx L(1)\approx P$.
\end{enumerate}
Finally, in case $a\leq b$ we have
\[
L(b,U(R(b,a)))=L(b,ULU(b',L(b,a)))=L(b,U(b',a))=L(a),
\]
thus  $\mathbf R$ satisfies operator divisibility.
\end{proof}

Also some kind of a converse of the previous result holds.

\begin{theorem}
Let $(P,\leq,{}',M,R,0,1)$ be an operator residuated poset satisfying operator divisibility and assume
\begin{align*}
M(x,y) & =L(U(x,y'),y), \\
R(x,y) & =LU(x',L(x,y))
\end{align*}
for all $x,y\in P$. Then $\mathbf P:=(P,\leq,{}',0,1)$ is a generalized orthomodular poset.
\end{theorem}

\begin{proof}
Let $a,b\in P$. We have
\begin{align*}
 M(x,x) & \approx L(U(x,x'),x)\approx L(x), \\
M(x',x) & \approx L(U(x',x'),x)\approx L(x,x').
\end{align*}
From $L(a')\subseteq R(a,0)$ we obtain $L(a,a')=M(a',a)\subseteq L(0)$, i.e.\ $L(a,a')=\{0\}$. From $L(a)\subseteq R(a,a'')$ we obtain $a\in L(a)=M(a,a)\subseteq L(a'')$ which implies $a\leq a''$. Conversely, From $L(a'')\subseteq R(a'',a)$ we obtain $a''\in L(a'')=M(a'',a'')\subseteq L(a)$ which implies $a''\leq a$. Together we have $a''=a$. Thus $'$ is an involution. Hence we can apply De Morgan's laws to $L(a,a')=\{0\}$ in order to obtain $U(a,a')=\{1\}$, i.e.\ $'$ is also a complementation on $(P,\leq)$ and hence $\mathbf P$ an orthoposet. If $a\leq b$ then using De Morgan's laws and operator divisibility we finally obtain
\begin{align*}
U(a,L(a',b)) & =(L(a',U(a,b')))'=(L(a',ULU(a,L(a',b'))))'=(L(a',U(R(a',b'))))'= \\
             & =(L(b'))'=U(b),
\end{align*}
i.e.\ $\mathbf P$ is a generalized orthomodular poset.
\end{proof}

\section{Directoids assigned to generalized orthomodular posets}

Let $\mathbf P=(P,\leq,{}',0,1)$ be a bounded poset with a unary operation $'$. An {\em algebra} $(P,\sqcup,{}',0,$ $1)$ of type $(2,1,0,0)$ is called {\em assigned to $\mathbf P$} if it satisfies the following conditions for all $x,y\in P$:
\begin{itemize}
\item $x\sqcup y=y$ if $x\leq y$,
\item $x\sqcup y=y\sqcup x\in U(x,y)$.
\end{itemize}
Let $\mathbf A=(A,\sqcup,{}',0,1)$ be an algebra of type $(2,1,0,0)$, abbreviate $(x'\sqcup y')'$ by $x\sqcap y$ and let $\mathbb G(\mathbf A)$ denote the ordered quintuple $(A,\leq,{}',0,1)$ where $\leq$ denotes the binary relation on $A$ defined by
\[
x\leq y\text{ if }x\sqcup y=y.
\]
If $\mathbf P=(P,\leq,{}',0,1)$ is a generalized orthomodular poset and $\mathbb A(\mathbf P)=(P,\sqcup,{}',0,1)$ an algebra assigned to $\mathbf P$ then $(P,\sqcup)$ is a directoid assigned to the poset $(P,\leq)$ and hence $\mathbb A(\mathbf P)$ will be called a {\em directoid assigned to $\mathbf P$}.

As promised in Section~2, we can characterize the class of all generalized orthomodular posets by means of assigned directoids. At first, we state the following Lemma.

\begin{lemma}\label{lem1}
Let $\mathbf P=(P,\leq,{}',0,1)$ be a generalized orthomodular poset, $a,b\in P$ and $\mathbb A(\mathbf P)=(P,\sqcup,{}',0,1)$ a directoid assigned to $\mathbf P$. Then
\begin{enumerate}[{\rm(i)}]
\item $(P,\sqcup)$ is a directoid,
\item $\mathbb G(\mathbb A(\mathbf P))=\mathbf P$,
\item $L(a,b)=\{(a\sqcap x)\sqcap(b\sqcap x)\mid x\in P\}=\{x\in P\mid(a\sqcap x)\sqcap(b\sqcap x)=x\}$,
\item $U(a,b)=\{(a\sqcup x)\sqcup(b\sqcup x)\mid x\in P\}=\{x\in P\mid(a\sqcup x)\sqcup(b\sqcup x)=x\}$.
\end{enumerate}
\end{lemma}

\begin{proof}
\
\begin{enumerate}
\item[(i)] and (ii) are already shown in Section~2. We prove the remaining assertions.
\item[(iii)] If $c\in P$ then
\begin{align*}
(a\sqcap c)\sqcap(b\sqcap c) & \leq a\sqcap c\leq a, \\
(a\sqcap c)\sqcap(b\sqcap c) & \leq b\sqcap c\leq b,
\end{align*}
i.e.\ $(a\sqcap c)\sqcap(b\sqcap c)\in L(a,b)$. If, conversely, $c\in L(a,b)$ then
\[
(a\sqcap c)\sqcap(b\sqcap c)=c\sqcap c=c.
\]
\item[(iv)] This follows from (iii) by duality.\end{enumerate}
\end{proof}

Condition (ii) of Lemma~\ref{lem1} shows that, though $\mathbb A(\mathbf P)$ is in general not uniquely determined by $\mathbf P$, it contains the whole information on $\mathbf P$, i.e.\ $\mathbf P$ can be reconstructed from $\mathbb A(\mathbf P)$.

Let $\mathbf P=(P,\leq,{}',0,1)$ be a generalized orthomodular poset and $\mathbb A(\mathbf P)=(P,\sqcup,{}',0,1)$ a directoid assigned to $\mathbf P$. Then we can easily check that the following conditions hold for all $x,y\in P$:
\begin{itemize}
\item $x\sqcap y=x$ if $x\leq y$,
\item $x\sqcap y=y\sqcap x\in L(x,y)$,
\item $(x\sqcup y)\sqcap x\approx x$,
\item $(x\sqcap y)\sqcup x\approx x$.
\end{itemize}
Moreover, $\sqcap$ satisfies identities (i) -- (iii) from Section~2 (with $\sqcup$ replaced by $\sqcap$) and hence $(P,\sqcap)$ is called a {\em {\rm(}meet-{\rm)}directoid}.

We now want to describe those directoids which are assigned to generalized orthomodular posets. We present here an easy characterization using three identities and one implication. The crucial thing is that we characterize posets by algebras with everywhere defined operations.

\begin{theorem}\label{th2}
Let $\mathbf P=(P,\leq,{}',0,1)$ be a bounded poset with a unary operation $'$ and $\mathbb A(\mathbf P)=(P,\sqcup,{}',0,1)$ an algebra assigned to $\mathbf P$. Then $\mathbf P$ is a generalized orthomodular poset if and only if $\mathbb A(\mathbf P)$ satisfies the following conditions:
\begin{enumerate}[{\rm(i)}]
\item if $(x\sqcup z)\sqcup(((x'\sqcap w)\sqcap((x\sqcup y)\sqcap w))\sqcup z)=z$ for all $w\in P$ then $(x\sqcup y)\sqcup z=z$,
\item $(x\sqcap y)\sqcup x\approx x$,
\item $(x\sqcup y)\sqcup(x'\sqcup y)\approx1$,
\item $x''\approx x$.
\end{enumerate}
\end{theorem} 

\begin{proof}
First assume $\mathbf P$ to be a generalized orthomodular poset. According to Lemma~\ref{lem1},
\[
L(x',x\sqcup y)=\{(x'\sqcap w)\sqcap((x\sqcup y)\sqcap w)\mid w\in P\}.
\]
Moreover,
\[
U(x,L(x',x\sqcup y))=\{z\in P\mid(x\sqcup z)\sqcup(w\sqcup z)=z\text{ for all }w\in L(x',x\sqcup y)\}.
\]
Now (i) follows from $U(x,L(x',x\sqcup y))\subseteq U(x\sqcup y)$. Identity (ii) follows from $x\sqcap y\leq x$ and identity (iii) from $(x\sqcup y)\sqcup(x'\sqcup y)\in U(x,x')=\{1\}$. Identity (iv) is evident. Conversely, assume $\mathbb A(\mathbf P)$ to satisfy (i) -- (iv). If we substitute $x$ and $y$ in (ii) by $x'$ and $y'$, respectively, and apply (iv) then we obtain that $x\leq y$ implies
\[
y'\leq y'\sqcup x'=(x\sqcup y)'\sqcup x'=(x''\sqcup y'')'\sqcup x'=(x'\sqcap y')\sqcup x'=x',
\]
i.e.\ $'$ is antitone. Because of (iv), $'$ is an involution. Altogether, $'$ is an antitone involution on $(P,\leq)$. If $y\in U(x,x')$ then $x,x'\leq z$, thus $x\sqcup z=z=x'\sqcup z$. Using (v) we compute
\[
y=y\sqcup y=(x\sqcup y)\sqcup(x'\sqcup y)=1
\]
proving $U(x,x')=\{1\}$ for all $x\in P$. Due to De Morgan's laws we have $L(x,x')=\{0\}$ showing that $x'$ is a complement of $x$. Summarizing, $\mathbf P=(P,\leq,{}',0,1)$ is an orthoposet. Finally, assume $x\leq y$. Then, obviously, $U(y)\subseteq U(x,L(x'y))$. If, conversely, $z\in U(x,L(x',y))$ then, since by (iii) of Lemma~\ref{lem1} we have
\[
(x'\sqcap t)\sqcap((x\sqcup y)\sqcap t)\in L(x',y)
\]
for all $t\in P$, we have
\[
(x\sqcup z)\sqcup(((x'\sqcap t)\sqcap((x\sqcup y)\sqcap t))\sqcup z)=z\sqcup z=z
\]
for all $t\in P$ whence by (i)
\[
z=(x\sqcup y)\sqcup z\in U(x\sqcup y)=U(y).
\]
This shows $U(x,L(x',y))\subseteq U(y)$, thus $U(x,L(x',y))=U(y)$. Hence, $\mathbf P$ is a generalized orthomodular poset.
\end{proof}

Consider the following identity:
\begin{enumerate}
\item[(i')] $x\sqcup y\leq(x\sqcup z)\sqcup((x'\sqcap(x\sqcup y))\sqcup z)$.
\end{enumerate}
Let
\begin{itemize}
\item $\mathcal A$ denote the class of all algebras $(P,\sqcup,{}',0,1)$ of type $(2,1,0,0)$ satisfying identities (ii) -- (iv) of Theorem~\ref{th2} as well as condition (i) of Theorem~\ref{th2},
\item $\mathcal W$ denote the variety of all algebras $(P,\sqcup,{}',0,1)$ of type $(2,1,0,0)$ satisfying identities (ii) -- (iv) of Theorem~\ref{th2} as well as identity (i').
\end{itemize}

\begin{lemma}
The class $\mathcal W$ is a subvariety of the class $\mathcal A$.
\end{lemma}

\begin{proof}
Let $\mathbf A=(P,\sqcup,{}',0,1)$ be a member of the variety $\mathcal W$ and assume that
\[
(x\sqcup z)\sqcup(((x'\sqcap w)\sqcap((x\sqcup y)\sqcap w))\sqcup z)=z\text{ for all }w\in P.
\]
Putting $w=1$ we derive the identity
\[
(x\sqcup z)\sqcup((x'\sqcap(x\sqcup y))\sqcup z)\approx z.
\]
Hence $x\sqcup y\leq z$ follows by (i'), i.e. $(x\sqcup y)\sqcup z=z$. This shows that (i) holds and therefore $\mathbf A$ belongs to the class $\mathcal A$.
\end{proof}

We are now interested in the congruence properties of the class $\mathcal A$ and the variety $\mathcal W$ (see e.g.\ \cite{CEL}). For this, we recall the following concepts.

Let $\mathcal C$ be a class of algebras of the same type and $\mathcal V$ a variety. Then the class $\mathcal C$ is called
\begin{itemize}
\item {\em congruence permutable} if $\Theta\circ\Phi=\Phi\circ\Theta$ for all $\mathbf A\in\mathcal C$ and $\Theta,\Phi\in\Con\mathbf A$,
\item {\em congruence distributive} if $(\Theta\vee\Phi)\wedge\Psi=(\Theta\wedge\Psi)\wedge(\Phi\wedge\Psi)$ for all $\mathbf A\in\mathcal C$ and $\Theta,\Phi,\Psi\in\Con\mathbf A$,
\item {\em arithmetical} if it is both congruence permutable and congruence distributive,
\item {\em congruence regular} if for each $\mathbf A=(A,F)\in\mathcal C$, every $a\in A$ and all $\Theta,\Phi\in\Con\mathbf A$ with $[a]\Theta=[a]\Phi$ we have $\Theta=\Phi$.
\end{itemize}
The following is well-known (cf.\ \cite{CEL}, Theorems~3.1.8, Corollary~3.2.4 and Theorem~6.1.3):
\begin{itemize}
\item The class $\mathcal C$ is congruence permutable if there exists a so-called {\em Maltsev term}, i.e.\ a ternary term $p$ satisfying
\[
p(x,x,y)\approx p(y,x,x)\approx y,
\]
\item The class $\mathcal C$ is congruence distributive if there exists a so-called {\em majority term}, i.e.\ a ternary term $m$ satisfying
\[
m(x,x,y)\approx m(x,y,x)\approx m(y,x,x)\approx x,
\]
\item The variety $\mathcal V$ is congruence regular if and only if there exists a positive integer $n$ and ternary terms $t_1,\ldots,t_n$ such that
\[
t_1(x,y,z)=\cdots=t_n(x,y,z)=z\text{ if and only if }x=y.
\]
\end{itemize}

\begin{theorem}
The class $\mathcal A$ is congruence distributive and the variety $\mathcal W$ is arithmetical and congruence regular.
\end{theorem}

\begin{proof}
First consider class $\mathcal A$. Define a ternary term $m$ via
\[
m(x,y,z):=(x\sqcup y)\sqcap(y\sqcup z)\sqcap(z\sqcup x).
\]
Then
\begin{align*}
m(x,x,z) & \approx(x\sqcup x)\sqcap(x\sqcup z)\sqcap(z\sqcup x)\approx x\sqcap(x\sqcup z)\approx x, \\
m(x,y,x) & \approx(x\sqcup y)\sqcap(y\sqcup x)\sqcap(x\sqcup x)\approx(x\sqcup y)\sqcap x\approx x, \\
m(x,z,z) & \approx(x\sqcup z)\sqcap(z\sqcup z)\sqcap(z\sqcup x)\approx(x\sqcup z)\sqcap z\approx z
\end{align*}
proving that $m$ is a majority term.

Now consider variety $\mathcal W$. We use the identity
\[
x\sqcup y\leq x\sqcup(x'\sqcap(x\sqcup y))
\]
which follows from (i') by putting $z=0$.

Define a ternary term $p$ via
\[
p(x,y,z):=(x\sqcup(y'\sqcap(y\sqcup z)))\sqcap(z\sqcup(y'\sqcap(y\sqcup x))).
\]
Then
\begin{align*}
p(x,x,z) & \approx(x\sqcup(x'\sqcap(x\sqcup z)))\sqcap(z\sqcup(x'\sqcap(x\sqcup x)))\approx(x\sqcup(x'\sqcap(x\sqcup z)))\sqcap z\approx z, \\
p(x,z,z) & \approx(x\sqcup(z'\sqcap(z\sqcup z)))\sqcap(z\sqcup(z'\sqcap(z\sqcup x)))\approx x\sqcap(z\sqcup(z'\sqcap(z\sqcup x)))\approx x.
\end{align*}
Thus $p$ is a Maltsev term.

Now put
\[
t(x,y):=(x'\sqcap(x\sqcup y))\sqcup(y'\sqcap(x\sqcup y)).
\]
Then
\[
t(x,x)\approx(x'\sqcap(x\sqcup x))\sqcup(x'\sqcap(x\sqcup x))\approx0,
\]
and if $t(x,y)=0$ then $x'\sqcap(x\sqcup y)=y'\sqcap(x\sqcup y)=0$ and hence
\begin{align*}
x & \leq x\sqcup y\leq x\sqcup(x'\sqcap(x\sqcup y))=x\sqcup0=x, \\
y & \leq x\sqcup y\leq y\sqcup(y'\sqcap(x\sqcup y))=y\sqcup0=y
\end{align*}
whence $x=x\sqcup y=y$. If, finally,
\begin{align*}
t_1(x,y,z) & :=t(x,y)\sqcup z, \\
t_2(x,y,z) & :=(t(x,y))'\sqcap z
\end{align*}
then
\begin{align*}
t_1(x,x,z) & \approx t(x,x)\sqcup z\approx z, \\
t_2(x,x,z) & \approx(t(x,x))'\sqcap z\approx z,
\end{align*}
and if $t_1(x,y,z)=t_2(x,y,z)=z$ then $t(x,y)\leq z\leq(t(x,y))'$ and hence $t(x,y)=t(x,y)\sqcap(t(x,y))'=0$ whence $x=y$.
\end{proof}

It is a question whether the class $\mathcal A$ is congruence permutable, too. We can establish a ``partial'' Maltsev term as follows. If
\[
p(x,y,z):=(x\vee L(y',y\sqcup z))\sqcap(z\vee L(y',x\sqcup y))
\]
then
\begin{align*}
p(x,x,z) & =(x\vee L(x',x\sqcup z))\sqcap(z\vee L(x',x\sqcup x))=(x\sqcup z)\sqcap(z\vee0)=z, \\
p(x,z,z) & =(x\vee L(z',z\sqcup z))\sqcap(z\vee L(z',x\sqcup z))=(x\vee0)\sqcap(x\sqcup z)=x.
\end{align*}
The problem, however, is that within the ``term'' $p$ there occurs the operator $L$ and, moreover, the suprema occurring in $p$ need not exist for all possible entries $x,y,z$, thus our ``term'' $p$ is only partial.

\section{Dedekind-MacNeille completion}

It is well-known that the Dedekind-MacNeille completion of an orthomodular poset need not be an orthomodular lattice. The aim of this section is to show that a bounded poset with a unary operation is a strong generalized orthomodular poset if and only if its Dedekind-MacNeille completion is nearly an orthomodular lattice.

The construction is as follows:

Let $\mathbf P=(P,\leq,{}',0,1)$ be a bounded poset with a unary operation. Define
\begin{align*}
 \DM(\mathbf P) & :=\{L(A)\mid A\subseteq P\}, \\
        A\vee B & :=LU(A,B)\text{ for all }A,B\in\DM(\mathbf P), \\
      A\wedge B & :=A\cap B\text{ for all }A,B\in\DM(\mathbf P), \\
            A^* & :=L(A')\text{ for all }A\in\DM(\mathbf P), \\
\BDM(\mathbf P) & :=(\DM(\mathbf P),\vee,\wedge,{}^*,\{0\},P).
\end{align*}
Then $\BDM(\mathbf P)$ is a complete lattice with a unary operation, called the {\em Dedekind-MacNeille completion} of $\mathbf P$, and $x\mapsto L(x)$ is an isomorphism from $(P,\leq,{}')$ to $(\{L(x)\mid x\in P\},\subseteq,{}^*)$.

An {\em orthomodular lattice} is a generalized orthomodular poset which is a lattice, i.e.\ which satisfies the {\em orthomodular law}
\[
x\leq y\text{ implies }y=x\vee(x'\wedge y)
\]
or, equivalently,
\[
x\leq y\text{ implies }x=y\wedge(x\vee y').
\]
$\BDM(\mathbf P)$ is said to be {\em nearly an orthomodular lattice} if
\[
L(a)\subseteq B\text{ implies }B=L(a)\vee(B\wedge(L(a))^*)
\]
for all $a\in P$ and $B\in\DM(\mathbf P)$. Note that if $\BDM(\mathbf P)$ is an orthomodular lattice then it is also nearly an orthomodular lattice, but our assumption is weaker since we do not quantify over all $A,B\in\DM(\mathbf P)$ with $A\subseteq B$, but only over all $a\in P$ and all subsets $B$ of $\DM(\mathbf P)$ with $L(a)\subseteq B$.

\begin{theorem}
Let $\mathbf P=(P,\leq,{}',0,1)$ be a bounded poset with a unary operation. Then $\mathbf P$ is a strong generalized orthomodular poset if and only if $\BDM(\mathbf P)$ is nearly an orthomodular lattice.
\end{theorem} 

\begin{proof}
It is easy to see that $\mathbf P$ is an orthoposet if and only if $\BDM(\mathbf P)$ is an ortholattice. Now let $a,b\in P$, $A,B\in\DM(\mathbf P)$ and $C\subseteq P$. Since
\[
L(b)\wedge(A\vee(L(b))^*)=L(b)\cap LU(A,L(b'))=L(b,U(A,b'))
\]
the following are equivalent:
\begin{align*}
                   & \mathbf P\text{ is a strong generalized orthomodular poset}, \\
        L(C)\leq b & \text{ implies }L(C)=L(b,U(L(C),b')), \\
   A\subseteq L(b) & \text{ implies }A=L(b,U(A,b')), \\
   A\subseteq L(b) & \text{ implies }A=L(b)\wedge(A\vee(L(b))^*), \\
L(b')\subseteq A^* & \text{ implies }A^*=L(b')\vee(A^*\wedge L(b)), \\
   L(a)\subseteq B & \text{ implies }B=L(a)\vee(B\wedge(L(a))^*), \\
                   & \BDM(\mathbf P)\text{ is nearly an orthomodular lattice}.
\end{align*}
\end{proof}

\begin{corollary}
Let $\mathbf P=(P,\leq,{}',0,1)$ be a bounded poset with a unary operation and assume that $\BDM(\mathbf P)$ is an orthomodular lattice. Then $\mathbf P$ is a strong generalized orthomodular poset.
\end{corollary}

Authors' addresses:

Ivan Chajda \\
Palack\'y University Olomouc \\
Faculty of Science \\
Department of Algebra and Geometry \\
17.\ listopadu 12 \\
771 46 Olomouc \\
Czech Republic \\
ivan.chajda@upol.cz

Helmut L\"anger \\
TU Wien \\
Faculty of Mathematics and Geoinformation \\
Institute of Discrete Mathematics and Geometry \\
Wiedner Hauptstra\ss e 8-10 \\
1040 Vienna \\
Austria, and \\
Palack\'y University Olomouc \\
Faculty of Science \\
Department of Algebra and Geometry \\
17.\ listopadu 12 \\
771 46 Olomouc \\
Czech Republic \\
helmut.laenger@tuwien.ac.at
\end{document}